\documentclass[12pt]{article}
\usepackage{amsmath}
\usepackage{amsfonts}
\usepackage{amssymb}
\newtheorem{theorem}{Theorem}[section]

\newtheorem{proposition}[theorem]{Proposition}

\newtheorem{remark}{Remark}
\newtheorem{note}[theorem]{Note}
\newtheorem{example}{Example}
\newtheorem{definition}{Definition}
\newenvironment{proof}[1][Proof]{\begin{trivlist}
\item[\hskip \labelsep {\bfseries #1}]}{\end{trivlist}}
\usepackage[a4paper]{geometry}
\usepackage{accents}
\usepackage{enumitem}
\makeatletter
\newcommand{\eqnum}{\leavevmode\hfill\refstepcounter{equation}\textup{\tagform@{\theequation}}}
\makeatother
\date{}
\begin{document}

\title{  Some Information Inequalities for Statistical Inference}
\author{Harsha K V and Alladi Subramanyam \\
 \normalsize Department of Mathematics, IIT Bombay, Mumbai, India \\\small Email: harshakv@math.iitb.ac.in ; as@math.iitb.ac.in}
\maketitle
\begin{abstract}
In this paper, we first describe the generalized notion of Cramer-Rao lower bound obtained by Naudts (2004) using two families of probability density functions, the original model and an escort model. We reinterpret the results in Naudts (2004) from a  statistical point of view and obtain some interesting examples in which this bound is attained. Further we obtain information inequalities which generalize the classical Bhattacharyya bounds in both regular and non-regular cases. 
\end{abstract}
\section{Introduction}
For every unbiased estimator $T$, an inequality of the type 
\begin{equation}
\mathrm{Var}_{\theta}(T) \geq d(\theta) 
\end{equation}
for every $\theta$ in the  parameter space $\Theta$, is called an information inequality and it plays an important role in parameter estimation. The early works of Cramer (1946) and Rao (1945) introduced the Cramer-Rao inequality for regular density functions.  For the non-regular density functions,  Hammersley (1950) and Chapman-Robbins (1951) introduced an inequality which come to be known as Hammersley-Chapman-Robbins inequality while Fraser and Guttman (1952) obtained the Bhattacharyya bounds. Later Vincze (1979) and Khatri (1980) introduced information inequalities by imposing the regularity assumptions on a  prior distribution rather than on the model. 

Recently in statistical physics, a generalized notion of Fisher information and a corresponding Cramer-Rao lower bound are introduced by Naudts (2004) using two families of probability density functions, the original model and an escort model.  Further he showed that in the case of a deformed exponential family of probability density functions, there exist an escort family and an estimator whose variance attains the bound. Also from an information geometric point of view, he obtained a dually flat structure of the deformed exponential family. 

In this article, concentrating on the statistical aspects of Naudts's paper we define several information inequalities which generalize the classical Hammersley-Chapman-Robbins bound and Bhattacharyya Bounds in both regular and non-regular cases. This is done by imposing the regularity conditions on the escort model rather than on the original model.

In Section 2, some preliminary results are stated. Section 3 describes the generalized Cramer-Rao lower bound obtained by Naudts (2004) reinterpreted from a statistical point of view and applied to many examples. Also we obtain many interesting examples in which the bound is optimal. In Section 4, we obtain a generalized notion of Bhattacharyya bounds in both regular and non-regular cases. We conclude with Discussions in Section 5.

\section{Preliminaries}
Let $X$ be a random vector with probability density function $f(x,\underaccent{\bar}{\theta})$, where $ \underaccent{\bar}{\theta} =(\theta_{1},\cdots,\theta_{p})^{\intercal} \in \Theta \in \Bbb R^{p}$ and $X$ takes values in $A\subseteq \Bbb R^{n}$.
To estimate a real valued function $\varphi$ of $\underaccent{\bar}{\theta}$, define a class of estimators as
\begin{eqnarray}
\mathcal{C}_{\varphi} =\lbrace S(X)  \mid  E_{f_{\underaccent{\bar}{\theta}}}(S(X))=\varphi(\underaccent{\bar}{\theta}), \forall \; \underaccent{\bar}{\theta} \in \Theta \rbrace.
\end{eqnarray}
Define 
\begin{eqnarray}
\mathcal{U}_{f} &= &\lbrace U(X) \mid E_{f_{\underaccent{\bar}{\theta}}}(U)=0 \; ; E_{f_{\underaccent{\bar}{\theta}}}(U^{2}) < \infty, \; \forall \; \underaccent{\bar}{\theta} \in \Theta \rbrace 
\end{eqnarray}
Let $\Psi=\lbrace \psi(x,\underaccent{\bar}{\theta}) \mid E_{f_{\underaccent{\bar}{\theta}}}(\psi)=0, \; E_{f_{\underaccent{\bar}{\theta}}}(\psi^{2}) < \infty, \; \mathrm{Cov}_{f_{\underaccent{\bar}{\theta}}}(U,\psi)=0, \; \forall \; U \in \mathcal{U}_{f}, \; \forall \; \underaccent{\bar}{\theta}  \rbrace$.\\
Let $S_{1}(x,\underaccent{\bar}{\theta}),\cdots,S_{m}(x,\underaccent{\bar}{\theta}) \in \Psi$. Let 
\begin{eqnarray}
E_{f_{\underaccent{\bar}{\theta}}}(TS_{i})=\lambda_{i}(\underaccent{\bar}{\theta}) ; \quad i=1, \cdots, m
\end{eqnarray}
where $\lambda_{i}$ is a real valued function of $\underaccent{\bar}{\theta}$.\\
Define
\begin{equation}
\psi(x,\underaccent{\bar}{\theta})=\sum_{i=1}^{m}\alpha_{i} S_{i}(x,\underaccent{\bar}{\theta}), \quad \alpha_{i} \in \Bbb R
\end{equation}
For any estimators $T, S \in \mathcal{C}_{\varphi}$, 
\begin{eqnarray}
\mathrm{Cov}_{f_{\theta}}(T,\psi)=\mathrm{Cov}_{f_{\theta}}(S,\psi)=\delta(\theta) \quad \mathrm{since} \quad T-S \in \mathcal{U}_{f}, \psi \in \Psi
\end{eqnarray}
Therefore $ \forall \; T \in \mathcal{C}_{\varphi}$, the Cauchy-Schwarz inequality
\begin{eqnarray}
\mathrm{Var}_{f_{\underaccent{\bar}{\theta}}} (T(x)) \geq \frac{(\mathrm{Cov}_{f_{\underaccent{\bar}{\theta}}}(T,\psi))^{2}}{\mathrm{Var}_{f_{\underaccent{\bar}{\theta}}} (\psi)} = \frac{\delta(\theta)^{2}}{\mathrm{Var}_{f_{\underaccent{\bar}{\theta}}} (\psi)}. \label{eq:1ab}
\end{eqnarray}
gives a lower  bound for the variance of all unbiased estimators of $\varphi(\underaccent{\bar}{\theta})$.\\
Now consider
\begin{eqnarray}
\mathrm{Var}_{f_{\underaccent{\bar}{\theta}}} (\psi) & = & \mathrm{Var}_{f_{\underaccent{\bar}{\theta}}} (\sum_{i=1}^{m}\alpha_{i} S_{i})= \alpha^{\intercal} \Sigma \alpha \\
(\mathrm{Cov}_{f_{\underaccent{\bar}{\theta}}}(T,\psi))^{2} &=& (\sum_{i=1}^{m}\alpha_{i} \lambda_{i}(\underaccent{\bar}{\theta}))^{2} =\alpha^{\intercal} MM^{\intercal} \alpha 
\end{eqnarray}
where $\alpha=(\alpha_{1},\cdots,\alpha_{m})^{\intercal} \in \Bbb R^{m}$, $\Sigma=(\Sigma_{ij})=(\mathrm{Cov}_{f}(S_{i},S_{j}))$ is the covariance matrix of $S=(S_{1},\cdots, S_{m})^{\intercal}$ and $M=(\lambda_{1}(\underaccent{\bar}{\theta}),\cdots, \lambda_{m}(\underaccent{\bar}{\theta}))^{\intercal}$.\\
Note that both $M$ and $\Sigma$ depends on $\underaccent{\bar}{\theta}$. But  for the convenience of writing, we suppress the index $\underaccent{\bar}{\theta}$. \\
Equation \eqref{eq:1ab} becomes
 \begin{eqnarray}
\mathrm{Var}_{f_{\underaccent{\bar}{\theta}}} (T(x)) \geq \frac{\alpha^{\intercal} M M^{\intercal} \alpha}{\alpha^{\intercal} \Sigma \alpha } \quad \forall \; \alpha \in \Bbb R^{m}
\end{eqnarray}
which implies
\begin{eqnarray}
\mathrm{Var}_{f_{\underaccent{\bar}{\theta}}} (T(x)) \geq \sup_{\alpha} \frac{\alpha^{\intercal} M M^{\intercal} \alpha}{\alpha^{\intercal} \Sigma \alpha } =M^{\intercal} \Sigma^{-1} M
\end{eqnarray}
where $\Sigma^{-1}$ is the inverse of the covariance matrix $\Sigma$.\\
For later use, we state the following well known theorem as 
\begin{proposition}{\label{prop1}}{\textbf{Information Inequality}}.
Let $X$ be a random vector with probability density function (pdf) $f(x,\underaccent{\bar}{\theta})$, where $\underaccent{\bar}{\theta}=(\theta_{1},\cdots,\theta_{p})^{\intercal} \in \Theta \in \Bbb R^{p}$. Consider an estimator $T(X) \in \mathcal{C}_{\varphi}$ , $S_{1}(x,\underaccent{\bar}{\theta}),\cdots,S_{m}(x,\underaccent{\bar}{\theta}) \in \Psi$ and the functions $\lambda_{i} :\Theta \rightarrow \Bbb R ; i=1, \cdots, m $  with 
\begin{eqnarray}
E_{f_{\underaccent{\bar}{\theta}}}(TS_{i})=\lambda_{i}(\underaccent{\bar}{\theta}) ; \quad i=1, \cdots, m
\end{eqnarray}
Then the variance of $T$ satisfies the inequality
\begin{eqnarray}
\mathrm{Var}_{f_{\underaccent{\bar}{\theta}}} (T(x)) \geq M^{\intercal} \Sigma^{-1} M \label{eq1a}
\end{eqnarray}
where $M=(\lambda_{1}(\underaccent{\bar}{\theta}),\cdots, \lambda_{m}(\underaccent{\bar}{\theta}))^{\intercal}$ and $\Sigma^{-1}$ is the inverse of the covariance matrix $\Sigma=(\Sigma_{ij})=(\mathrm{Cov}_{f_{\underaccent{\bar}{\theta}}}(S_{i},S_{j}))$.
The equality in \eqref{eq1a} holds iff  
\begin{equation}
S^{\intercal} \Sigma^{-1} M =a(\underaccent{\bar}{\theta}) (T(x)-\varphi(\underaccent{\bar}{\theta})) \label{eq1b}
\end{equation}
 for some function $a(\underaccent{\bar}{\theta})$ and $S=(S_{1},\cdots,S_{m})^{\intercal}$
\end{proposition}
\section{Generalized Cramer-Rao Type Lower Bound}
Naudts (2004) introduced a generalized notion of  Fisher information by replacing the original model by an escort model at suitable places. Using this, he obtained a generalized Cramer-Rao lower bound. To study the statistical implications of this generalization, first we reinterpret Naudts's generalized as follows.\\
Let $g(x,\underaccent{\bar}{\theta})$ be any density function parametrized by $\underaccent{\bar}{\theta}=(\theta_{1},\cdots,\theta_{p})^{\intercal} \in \Theta \in \Bbb R^{p}$. Define
\begin{eqnarray}
\mathcal{U}_{g} &= &\lbrace U(X) \mid E_{g_{\underaccent{\bar}{\theta}}}(U)=0 \; ; E_{g_{\underaccent{\bar}{\theta}}}(U^{2}) < \infty \; \forall \; \underaccent{\bar}{\theta} \in \Theta \rbrace
\end{eqnarray}
Let us make the following assumptions,
\begin{enumerate}[label=(\alph*)]
\item The probability measure $P_{g} $ is absolutely continuous with respect to the probability measure $P_{f}.$\eqnum\label{myeqn1}
\item $\mathcal{U}_{f} \subseteq \mathcal{U}_{g}$. \eqnum\label{myeqn2}
\end{enumerate}
\begin{remark} \label{rem1}
If $T$ is a complete statistic, then clearly $\mathcal{U}_{f} \subseteq \mathcal{U}_{g}$.
\end{remark}
Naudts (2004) defined a generalized Fisher information $N(\underaccent{\bar}{\theta})=(N_{ij}(\underaccent{\bar}{\theta}))$ as 
 \begin{eqnarray}
 N_{ij}(\underaccent{\bar}{\theta})= \int \partial_{i} g(x,\underaccent{\bar}{\theta}) \partial_{j} g(x,\underaccent{\bar}{\theta}) \frac{1}{f(x,\underaccent{\bar}{\theta})} dx \; ;\quad \partial_{i}:=\frac{\partial}{\partial \theta_{i}}  \; \mathrm{and} \; \; i,j=1,\cdots, p   
 \end{eqnarray}
 Note that when $f=g$, $N(\underaccent{\bar}{\theta})$ reduces to the Fisher information $I(\underaccent{\bar}{\theta})$.
 \begin{theorem} Let $X$ be a random vector with pdf $f(x,\underaccent{\bar}{\theta})$.  Let $g(x,\underaccent{\bar}{\theta})$ be a pdf satisfying \eqref{myeqn1} $  \& $ \eqref{myeqn2}. Assume that
 \begin{enumerate}[label=(\alph*)]
\item  $ \partial_{i} g(x,\underaccent{\bar}{\theta}) $ exists for all $x \in A$ and $\underaccent{\bar}{\theta} \in \Theta$, where $i=1,\cdots, p$  \eqnum\label{myeqn3}
\item $0< N_{ij}(\underaccent{\bar}{\theta}) < \infty$ and $N(\underaccent{\bar}{\theta})$ is non-singular. \eqnum\label{myeqn4}
\item partial derivatives of functions of $\underaccent{\bar}{\theta}$ expressed as integrals with respect to $g(x,\underaccent{\bar}{\theta})$ can be obtained by differentiating under the integral sign. \eqnum\label{myeqn5}
\end{enumerate}
Then for $T(X) \in \mathcal{C}_{\varphi}$, the variance of $T$ satisfies  
 \begin{eqnarray}
\mathrm{Var}_{f_{\underaccent{\bar}{\theta}}}(T(X)) \geq   M^{\intercal} N^{-1}(\underaccent{\bar}{\theta}) M \label{eq:1d}
\end{eqnarray}
where $E_{g_{\underaccent{\bar}{\theta}}}[T]=\lambda(\underaccent{\bar}{\theta})$ and $M=(\partial_{1}\lambda(\underaccent{\bar}{\theta}),\cdots, \partial_{p}\lambda(\underaccent{\bar}{\theta})^{\intercal}$. 
\end{theorem}
\begin{proof} From Proposition \ref{prop1},  choose  $m=p$ functions $S_{i} \in  \Psi$
\begin{equation}
S_{i}=\frac{\partial_{i}g(x,\underaccent{\bar}{\theta})}{f(x,\underaccent{\bar}{\theta})} ,  \quad i=1,\cdots,p
\end{equation}
It is easy to see that $E_{f_{\underaccent{\bar}{\theta}}}(TS_{i})=\partial_{i} \lambda(\underaccent{\bar}{\theta})$, where $i=1,\cdots,p$. Applying Proposition \ref{prop1}, the bound in Equation \eqref{eq:1d} is obtained. The fact that $\mathcal{U}_{f} \subseteq \mathcal{U}_{g}$ ensures that the bound is same for all unbiased estimators $T$ of $\varphi(\underaccent{\bar}{\theta})$. 
$ \hfill \blacksquare $
\end{proof}
Now we give some of the interesting examples in which the Naudt's generalized Cramer-Rao bound is optimal. 
 \begin{example}
Suppose $Y_{1},\cdots,Y_{n}$ are independent uniform random variables in $[0,\theta]$, where $\theta>0$. Then $X = \max \lbrace {Y_{1},\cdots,Y_{n}} \rbrace$ has a pdf  
\begin{equation}
f(x,\theta)=\frac{n x^{n-1}}{\theta^{n-1}}, \quad x \geq \theta
\end{equation}
Now consider an unbiased estimator  $T(X)=\frac{(n+1)X}{n}$ of $\theta$. Then 
\begin{eqnarray}
\mathrm{Var}_{f_{\theta}}(T) &=& \frac{\theta^{2}}{n(n+2)}
\end{eqnarray}
Consider a pdf $g(x,\theta)$ as
\begin{equation}
g(x,\theta)=\frac{n(n+1)(1-\frac{x}{\theta}) x^{n-1}}{\theta^{n}}.
\end{equation}
Using Remark \ref{rem1}, clearly $\mathcal{U}_{f} \subseteq \mathcal{U}_{g}$. Now
\begin{equation}
E_{g_{\theta}}[T] =\lambda(\theta)=\frac{(n+1)\theta}{n+2} \quad \mathrm{and} \quad N(\theta)=\frac{n(n+1)^{2}}{(n+2)\theta^{2}}
\end{equation}
The lower bound in Equation \eqref{eq:1d} is obtained as
\begin{eqnarray}
 \frac{(\lambda'(\theta))^{2}}{N(\theta)} = \frac{\theta^{2}}{n(n+2)}=Var_{f_{\theta}}(T) 
\end{eqnarray}
Thus the estimator $T(X)$ is an unbiased estimator of $\theta$ which attains the generalized Cramer Rao bound by Naudts. When $n=1$, this example reduces to Example 1 given in Naudts (2004). Note that in this case, $\mathrm{Var}_{f_{\theta}}(T)$ does not attain the Hammersley-Chapman-Robbins lower bound.
\end{example}
\begin{example} Suppose $Y_{1},\cdots,Y_{n}$ are independent random variables, 
\begin{equation}
Y_{1},\cdots,Y_{n} \sim \exp ( -(y-\theta)),\; y \geq \theta, \theta >0.
\end{equation}
Then the random variable $X  =\min \lbrace {Y_{1},\cdots,Y_{n}} \rbrace$ has a pdf
\begin{equation}
f(x,\theta)=n \exp(-n(x-\theta)), \; x \geq \theta
\end{equation}
Now consider an unbiased estimator $T(X)=X-\frac{1}{n}$ of $\theta$. Then 
\begin{eqnarray}
Var_{f_{\theta}}(T) &=& \frac{1}{n^{2}}
\end{eqnarray}
Then the pdf $g(x,\theta)$ which optimizes the bound in Equation \eqref{eq:1d} is 
\begin{equation}
g(x,\theta)=n^{2}(x-\theta) \exp(-n(x-\theta)), x \geq \theta
\end{equation}
Using Remark \ref{rem1}, clearly $\mathcal{U}_{f} \subseteq \mathcal{U}_{g}$. Note that $E_{g_{\theta}}[T] =\lambda(\theta)=\frac{1}{n} +\theta$ and the bound in Equation \eqref{eq:1d} is obtained as
\begin{eqnarray}
 \frac{(\lambda'(\theta))^{2}}{N(\theta)}=\frac{1}{n^{2}}=Var_{f_{\theta}}(T) 
\end{eqnarray}
\end{example}
 \begin{example}\textbf{Location family} \\
Let $f(x)$ and $g(x)$ be two density functions on $x\in D'\subseteq \Bbb R$ satisfying \eqref{myeqn1} $  \& $ \eqref{myeqn2}. Now let $X$ be a random variable with density function $f(x,\theta)=f(x-\theta), \theta \in \Bbb R$ and $x\in D\subseteq \Bbb R$. 
Let $g(x,\theta)=g(x-\theta)$. Let $T(X)$ be an unbiased estimator for $\varphi(\theta)$. Let $E_{g}(T)=\lambda(\theta)$.  Then from Equation \eqref{eq1b}, the optimality condition for the  bound in Equation  \eqref{eq:1d} is given by
\begin{eqnarray}
\frac{\partial_{\theta} g(x,\theta)}{f(x,\theta)} =a(\theta)(T(x)-\varphi(\theta)) \label{eq1c}
\end{eqnarray}
for some function $a(\theta)$. In this case
\begin{equation}
\partial_{\theta} g(x,\theta)=-g'(x - \theta)
\end{equation}
where $g'$ denote the derivative of $g(x)$ with respect to  $x$. Then \eqref{eq1c} becomes  
\begin{eqnarray}
 g'(x-\theta) =a(\theta)(\varphi(\theta)-T(x))f(x,\theta)
\end{eqnarray}
Let $\theta=0$ and  $x_{0} \in D'$,  then
\begin{eqnarray}
g(x) &=& a(0) \left( \varphi(0) \int_{x_{0}}^{x} f(x) dx - \int_{x_{0}}^{x} T(x)f(x) dx \right) \\
&=&a(0)h(x)
\end{eqnarray}
where 
\begin{equation}
h(x)= \varphi(0) \int_{x_{0}}^{x} f(x) dx - \int_{x_{0}}^{x} T(x)f(x) dx <\infty
\end{equation}
can be computed since $f(x), T(x), \varphi(0)$ are given.\\
Now $a(0)$ can be solved from the normalization condition $\int_{D'} g(x) dx=1$ as 
\begin{eqnarray}
a(0)=\frac{1}{\int_{D'}h(x) dx} \quad \mathrm{if} \quad \int_{D'}h(x) dx <\infty
\end{eqnarray}
Thus the optimizing family $g(x,\theta)=g(x-\theta)$ is obtained.
\end{example}
\begin{example} \textbf{Scale family} \\
 Let $f(x)$ and $g(x)$ be two density functions on $x\in D'\subseteq \Bbb R$ satisfying \eqref{myeqn1} $  \& $ \eqref{myeqn2}. Now let
\begin{eqnarray}
X \sim f(x,\theta)=\frac{1}{\theta}f(\frac{x}{\theta}) \quad x\in D\subseteq \Bbb R, \theta>0
\end{eqnarray}
and \begin{equation}
g(x,\theta)=\frac{1}{\theta}g(\frac{x}{\theta})
\end{equation}
Let $T(X)$ be an unbiased estimator for $\varphi(\theta)$. Let $E_{g}(T)=\lambda(\theta)$.  Then from \eqref{eq1b},
\begin{eqnarray}
\frac{\partial_{\theta} g(x,\theta)}{f(x,\theta)} =a(\theta)(T(x)-\varphi(\theta)) 
\end{eqnarray}
for some function $a(\theta)$.\\
\begin{eqnarray}
\frac{-x}{\theta^{3}} g'(x/\theta)-\frac{1}{\theta^{2}} g(x/\theta)=a(\theta)(T(x)-\varphi(\theta)) f(x,\theta)
\end{eqnarray}
where $g'$ denotes the derivative of function $g(x)$ with respect to $x$. \\
Let $\theta=1$.  Then we have
\begin{eqnarray}
x g'(x)+g(x)=a(1)(\varphi(1)-T(x)) f(x)
\end{eqnarray}
Let $x_{0} \in D'$. Integrating the above equation from $x_{0}$ to $x$, we get
\begin{eqnarray}
xg(x)-x_{0}g(x_{0}) &=&a(1) \int_{x_{0}}^{x} (\varphi(1)-T(x)) f(x) dx \\
&=&a(1)(h(x)-h(x_{0}))
\end{eqnarray}
where 
\begin{equation}
h(x)-h(x_{0})=\int_{x_{0}}^{x} (\varphi(1)-T(x)) f(x) dx
\end{equation}
Thus we get
\begin{eqnarray}
xg(x)=a(1)h(x) \Rightarrow g(x)=a(1)k(x)
\end{eqnarray}
for some function $k(x)$. \\
Now $a(1)$ can be solved from the normalization condition of the function $\int_{D'}g(x)dx=1$ as 
\begin{eqnarray}
a(1)=\frac{1}{\int_{D'}k(x) dx} \quad \mathrm{if} \quad \int_{D'}k(x) dx <\infty
\end{eqnarray}
Thus the optimizing family $g(x,\theta)=\frac{1}{\theta}g(\frac{x}{\theta})$ is obtained.
\end{example}
\begin{example}
Suppose $Y_{1},\cdots,Y_{n}$ are independent uniform random variables in $[0,\theta]$, where $\theta>0$. Then
\begin{equation}
X=\max \lbrace {Y_{1},\cdots,Y_{n}}\rbrace  \sim f(x,\theta)=\frac{n x^{n-1}}{\theta^{n-1}}
\end{equation}
Now consider an unbiased estimator  $T(X)=\frac{(n+k)X^{k}}{n}$ for $\theta^{k}$, where $k \geq 1$. Then 
\begin{eqnarray}
Var_{f_{\theta}}(T) &=& \frac{k^{2}\theta^{2k}}{n(n+2k)^{2}}
\end{eqnarray}
Now define a pdf $g(x,\theta)$ as 
\begin{equation}
g(x,\theta)=\frac{n(n+k)(1-\frac{x^{k}}{\theta^{k}}) x^{n-1}}{k\theta^{n}}
\end{equation}
Using Remark \ref{rem1}, clearly $\mathcal{U}_{f} \subseteq \mathcal{U}_{g}$. Then the bound in Equation \eqref{eq:1d} is obtained as
\begin{eqnarray}
 \frac{(\lambda'(\theta))^{2}}{N(\theta)}= \frac{k^{2}\theta^{2k}}{n(n+2k)^{2}}=Var_{f_{\theta}}(T) 
\end{eqnarray}
Thus the estimator $T(X)$ is an unbiased estimator of $\theta^{k}$ which attains the bound in Equation \eqref{eq:1d}.
\end{example}
\begin{example} Let $f(x,\theta)$ be the Gamma distribution with a scale parameter $\theta>0$ and a known shape parameter $\alpha >0$, 
\begin{eqnarray}
f(x,\theta)=\frac{1}{\Gamma(\alpha) } \frac{x^{\alpha-1}e^{-x/\theta}}{\theta^{\alpha}} 
\end{eqnarray}
Let $T(X)=\frac{\Gamma(\alpha)}{\Gamma(\alpha+k)} X^{k}$, where $k$ is an integer such that $k \neq 0$ and $2k+\alpha>0$. Then $T$ is an unbiased estimator of $\theta^{k}$ with $E_{f_{\theta}}(T^{2})<\infty$. 
\begin{eqnarray}
\mathrm{Var}_{f_{\theta}}(T) & = & \left[  \frac{\Gamma(\alpha) \Gamma(2k+\alpha)}{(\Gamma(\alpha+k))^{2}} -1 \right] \theta^{2k}
\end{eqnarray}
Consider a pdf $g(x,\theta)$ such that $T$ attains the bound in Equation \eqref{eq:1d} as follows. \\
For $k>0$, 
\begin{equation}
g(x,\theta)=\frac{1}{c} \frac{ e^{-x/\theta}}{\theta} \left[ \sum_{i=0}^{k-1}s_{i} \; (\frac{x}{\theta})^{\alpha+k-(i+2)}  \right],  \quad c=\sum_{i=0}^{k-1} s_{i} \Gamma(\alpha+k-(i+1))
\end{equation}
where $s_{i}=\prod_{j=1}^{i} (\alpha+k-j); i=1,\cdots,k-1$ and $s_{0}=1$.\\
For $k<0$ and $k \neq -1$,
\begin{equation}
g(x,\theta)=\frac{1}{c} \frac{ e^{-x/\theta}}{\theta} \left[ \sum_{i=1}^{k_{1}}s_{i} \; (\frac{x}{\theta})^{\alpha-(i+1)}  \right],  \quad c=\sum_{i=0}^{k_{1}} s_{i} \Gamma(\alpha-i)
\end{equation}
where $k_{1}=-k$, $s_{i}=\prod_{j=1}^{i-1} (\alpha-j); i=2,\cdots,k_{1}$ and $s_{0}=1$.\\
For $k =-1$,
\begin{equation}
g(x,\theta)=\frac{1}{\Gamma(\alpha-1)}\frac{ x^{\alpha-2} e^{-x/\theta}}{\theta^{\alpha-1}}
\end{equation}
This is an interesting special case as $T=1/X$ does not attain the Bhattacharyya bounds of any order while it attains the bound in Equation \eqref{eq:1d}.  
\end{example}
\begin{example} Consider the Normal distribution $\mathcal{N}(0,\theta^{2})$ given by
\begin{eqnarray}
 f(x,\theta)=\frac{1}{\sqrt{2 \pi} \theta} e^{\frac{-x^{2}}{2\theta^{2}}} , \quad  x \in \Bbb R \quad  \mathrm{and } \quad \theta >0
 \end{eqnarray}
 Consider an unbiased estimator $T(X)=\frac{X^{4}}{3}$ for  $\theta^{4}$. Then $ \mathrm{Var}_{f_{\theta}}(T) =\frac{32 \theta^{8}}{3}$. Consider a pdf 
 \begin{eqnarray}
 g(x,\theta)=\frac{1}{\sqrt{2 \pi} \theta} \left(  \frac{3}{4} + \frac{x^{2}}{4 \theta^{2}}\right)  e^{\frac{-x^{2}}{2\theta^{2}}}
 \end{eqnarray}
 Note that 
 \begin{eqnarray}
  N(\theta)=\frac{6}{\theta^{2}} \quad \mathrm{and} \quad \lambda(\theta)= E_{g_{\theta}}(T)=2 \theta^{4}
  \end{eqnarray}
  Thus the bound in Equation \eqref{eq:1d} is obtained as
  \begin{eqnarray}
   \frac{(\lambda'(\theta))^{2}}{N(\theta)}=\frac{32 \theta^{8}}{3}=\mathrm{Var}_{f_{\theta}}(T)
  \end{eqnarray}
  Thus $T$ attains Naudts's bound with optimizing family $g(x,\theta)$. Note that $f(x,\theta)$ belongs to exponential family and $T(X)$ is a second degree polynomial in the canonical statistic $X^{2}$. Hence it attains the Bhattacharyya bound of order $2$. Thus the `first order' bound obtained using $g$ is equal to the second order Bhattacharyya bound. 
\end{example}
 \begin{example} \textbf{Poisson distribution}\\
 Let $X_{1},\cdots,X_{n}$ are i.i.d  random variables from Poisson distribution 
 \begin{eqnarray}
 f(x,\theta)=\frac{\theta^{x} e^{-\theta}}{x !}, \quad x=0,1, \cdots \quad \mathrm{and} \quad \theta >0.
 \end{eqnarray}
 Consider the joint pdf 
 \begin{eqnarray}
 f(x_{1},\cdots, x_{n},\theta)=\frac{\theta^{n \bar{x}} e^{-n \theta}}{x_{1}! \cdots x_{n}!}, \quad \mathrm{where} \quad \bar{x}= \frac{x_{1}+\cdots+x_{n}}{n}.
 \end{eqnarray}
Consider an unbiased estimator $T(X)=\bar{X} (\bar{X}-\frac{1}{n})$ for $\theta^{2}$. 
$T$ attains the bound in Equation \eqref{eq:1d} if we choose the pdf 
\begin{eqnarray}
g(x_{1},\cdots, x_{n},\theta)=\frac{1}{2} \frac{\theta^{n \bar{x}} e^{-n \theta}}{x_{1}! \cdots x_{n}!} + \frac{\bar{x}}{2}  \frac{  \theta^{n \bar{x}-1} e^{-n \theta}}{x_{1}! \cdots x_{n}!}.
\end{eqnarray}
Note that $\mathrm{Var}_{f_{\theta}}(T)$ attains the Bhattacharyya bound of order $2$ while it attains `first order' Naudts's bound. 
 \end{example}
 \begin{example} Let $X_{1},\cdots,X_{n}$ are i.i.d uniform random variables in $[0,\theta]$, where $\theta>0$.  Then the joint pdf is
 \begin{equation}
 f(x_{1},\cdots,x_{n}, \theta)= \frac{1}{\theta^{n}} \;  \Pi_{i=1}^{n} \mathbf{1}_{ \{0 \leq x_{i} \leq \theta \} } 
 \end{equation}
 where $\mathbf{1}$ denotes the indicator function.\\
 Note that $T = \max \lbrace {X_{1},\cdots,X_{n}} \rbrace$  is a sufficient statistic with $E_{f_{\theta}}(T)=\frac{n}{n+1} \theta$ and $\mathrm{Var}_{f_{\theta}}(T)$ attains the bound in Equation \eqref{eq:1d} if we choose the pdf
 \begin{eqnarray}
 g(x_{1},\cdots,x_{n}, \theta)=\frac{n+1}{\theta^{n}} (1-\frac{t}{\theta}); \; 0 \leq t \leq \theta \quad \mathrm{where} \quad t=\max \lbrace {x_{1},\cdots,x_{n}} \rbrace.
 \end{eqnarray}
Note that $g(x_{1},\cdots,x_{x}, \theta)$ can be written as 
\begin{eqnarray}
g(x_{1},\cdots,x_{x},\theta)=Z\left( \frac{n+1}{\theta^{n}} -\frac{(n+1)t}{\theta^{n+1}}-1 \right) 
\end{eqnarray}
where the $Z$ is a function defined by  $Z(u)=[1+u]_{+}$, with $[v]_{+}=\max \{v,0 \}$ and $F(u)=u-1$ is the inverse function of $Z$. \\
Such family  $\lbrace g(x,\theta) \vert \theta \in \Theta \rbrace$ is called a deformed exponential family with a deformed logarithm function $F$  and  deformed exponential function  $Z$ (refer Naudts(2004) for more details). From the Proposition 5.2, Naudts (2004), it can be easily seen that $f(x,\theta)$ is the $F$-escort distribution so that the variance of the sufficient statistic $T$ attains the Naudts's bound.  
 \end{example}
 \begin{remark}{\label{remdefor}} Deformed exponential family is a generalization of exponential family in which the deformed logarithm of the density function is a linear function of the statistic $T$. In exponential family the statistic $T$ is sufficient and complete under some conditions. As in exponential family, $T$ is sufficient in deformed exponential family also. For statistical applications, the definition of deformed exponential family should include the requirement that $T$ is a complete statistic. \\
 In the above example, $g$ is a deformed exponential family while this is not the case in most of the other examples. However, $\mathrm{Var}_{f_{\theta}}(T)$ attains the bound given by Naudts (2004). 
 \end{remark}
\section{Generalized Bhattacharyya Bounds}
In this section, we obtain an information inequality which generalizes the Bhattacharyya bound given by Fraser and Guttman (1952). This is defined using the divided difference of a density function $g(x,\theta)$ satisfying the conditions \eqref{myeqn1} $  \& $ \eqref{myeqn2}. We begin by recalling the definition of the divided difference formula.
\subsection{One parameter case}
\begin{definition}Let $h(\theta)$ be a scalar function of $\theta \in \Theta \subseteq  \Bbb R$.  Let $k \geq  1$ be a positive integer.  Let us define the divided difference of the function $h$ at $k+1$ nodes $\theta^{0},\cdots, \theta^{k}$. We have $k+1$ data points,
\begin{eqnarray}
(\theta^{0}, h(\theta^{0})), \cdots, (\theta^{k}, h(\theta^{k}))
\end{eqnarray}
Define the first divided difference of $h$ as
\begin{eqnarray}
\underset{\theta^{\nu+1}}{\Delta} h(\theta^{\nu})  &:=& \frac{h(\theta^{\nu+1})-h(\theta^{\nu})}{\theta^{\nu+1}-\theta^{\nu}}; \;\;\nu =0 ,\cdots k-1 
\end{eqnarray}
Second divided difference is given by
\begin{eqnarray}
\underset{\theta^{\nu+1}, \theta^{\nu+2}}{\Delta^{2}} h(\theta^{\nu}) & := & \underset{\theta^{\nu+2}}{\Delta} \; \underset{\theta^{\nu+1}}{\Delta}  h(\theta^{\nu})= \frac{\underset{\theta^{\nu+2}}{\Delta} h(\theta^{\nu +1})-\underset{\theta^{\nu+1}}{\Delta} h(\theta^{\nu})}{\theta^{\nu+2}-\theta^{\nu}} \\
& & \; \mathrm{where} \;  \nu =0 ,\cdots k-2  \nonumber 
\end{eqnarray}
In general, for $j \geq 2$, the $j^{\mathrm{th}}$-divided difference is defined as
\begin{eqnarray}
\underset{\theta^{\nu+1}, \cdots, \theta^{\nu+j}}{\Delta^{j}} h(\theta^{\nu}) & := & \underset{\theta^{\nu+j}}{\Delta} \cdots \underset{\theta^{\nu+1}}{\Delta}  h(\theta^{\nu})= \frac{\underset{\theta^{\nu+2}, \cdots, \theta^{\nu+j}}{\Delta^{j-1}} h(\theta^{\nu +1})- \underset{\theta^{\nu+1}, \cdots, \theta^{\nu+j-1}}{\Delta^{j-1}} h(\theta^{\nu})}{\theta^{\nu+j}-\theta^{\nu}} \nonumber \\
& & \; \mathrm{where} \; \nu= 0 \cdots k-j;
\end{eqnarray}
\end{definition}
Choose and fix $\theta^{0}$ in $\Theta$.  For convenience, we write $g_{x}(\theta)$ instead of $g(x,\theta)$. Let $T(X)$ be an unbiased estimator of a real valued function $\varphi(\theta)$ of $\theta$. Then consider $i^{\mathrm{th}}$ divided difference of the density $g(x,\theta)$ on $k+1$ nodes of $\theta^{0},\cdots,\theta^{k}$, where $i=1,\cdots,k$. Define 
\begin{eqnarray}
S_{i}= \frac{1}{f(x,\theta^{0})} \underset{\theta^{1}, \cdots, \theta^{i}}{\Delta^{i}} g_{x}(\theta^{0}), \; i=1,\cdots k
\end{eqnarray}
We now give a lower bound for the variance of $T$ using these functions. 
\begin{theorem}{\label{nonreg}}
Let $g(x,\theta)$ be a density function satisfying conditions \eqref{myeqn1} $  \& $ \eqref{myeqn2}  with  $E_{g_{\theta^{0}}}[T]=\lambda(\theta^{0})$. For $T(X) \in \mathcal{C}_{\varphi}$, the variance of $T$ satisfies  
 \begin{eqnarray}
\mathrm{Var}_{f_{\theta^{0}}}(T(X)) \geq  \sup_{\theta^{1},\cdots, \theta^{k}}  M^{\intercal} \Sigma^{-1} M \label{eq:1e}
\end{eqnarray}
where $M=\left(  \underset{\theta^{1}}{\Delta} \lambda(\theta^{0}),\cdots,  \underset{\theta^{1}, \cdots, \theta^{i}}{\Delta^{k}} \lambda(\theta^{0}) \right) ^{\intercal}$,  $ \underset{\theta^{1}, \cdots, \theta^{i}}{\Delta^{i}} \lambda(\theta^{0})$ is the $i^{\mathrm{th}}$ divided difference of $\lambda$, $i=1,\cdots,k$ and  $\Sigma=(\Sigma_{ij})$ is the covariance matrix of the column vector $S=(S_{1},\cdots,S_{k})^{\intercal}$.
\end{theorem}
\begin{proof}
Note that 
\begin{eqnarray}
E_{f_{\theta^{0}}}[S_{i}] & = & \int  \underset{\theta^{1}, \cdots, \theta^{i}}{\Delta^{i}} g_{x}(\theta^{0}) dx\\
&=& \int \left( \sum_{j=0}^{i} \frac{g_{x}(\theta^{j})}{\prod_{l \neq j}(\theta^{j}-\theta^{l})} \right) dx \\
&=& \sum_{j=0}^{i}  \frac{1}{\prod_{l \neq j}(\theta^{j}-\theta^{l})} =0
\end{eqnarray}
Also we have
\begin{eqnarray}
E_{f_{\theta^{0}}}[T S_{i}] & =& \int T(x)  \underset{\theta^{1}, \cdots, \theta^{i}}{\Delta^{i}} g_{x}(\theta^{0}) dx \\
&= & \int T(x) \left( \sum_{j=0}^{i} \frac{g_{x}(\theta^{j})}{\prod_{l \neq j}(\theta^{j}-\theta^{l})} \right) dx \\
& = & \sum_{j=0}^{i} \frac{1}{\prod_{l \neq j}(\theta^{j}-\theta^{l})} \int T(x) g_{x}(\theta^{j}) dx\\
& = & \sum_{j=0}^{i} \frac{\lambda(\theta^{j})}{\prod_{l \neq j}(\theta^{j}-\theta^{l})} =   \underset{\theta^{1}, \cdots, \theta^{i}}{\Delta^{i}} \lambda(\theta^{0})
\end{eqnarray}
where $  \underset{\theta^{1}, \cdots, \theta^{i}}{\Delta^{i}} \lambda(\theta^{0})$ is the $i^{\mathrm{th}}$ divided difference of the function $\lambda$. \\
Hence it follows that $S_{i} \in \Psi, \quad i=1,\cdots, k$ .  Apply Proposition \ref{prop1} for 
 $S_{i}$ to obtain the bound in Equation \eqref{eq:1e}  with $M=\left(  \underset{\theta^{1}}{\Delta} \lambda(\theta^{0}),\cdots,  \underset{\theta^{1}, \cdots, \theta^{i}}{\Delta^{k}} \lambda(\theta^{0}) \right) ^{\intercal}$.  \hfill $\blacksquare$
\end{proof}
Let us define 
\begin{eqnarray}
S_{i}=\frac{ g^{i}(x,\theta)}{f(x,\theta)}, \; i=1,\cdots k.
\end{eqnarray}
where $g^{i}(x,\theta)$ denote the $i^{\mathrm{th}}$ derivative of $g(x,\theta)$ with respect to $\theta$.
\begin{theorem} 
Let $X$ be a random vector with pdf $f(x,\theta)$.  Let $g(x,\theta)$ be a pdf satisfying \eqref{myeqn1} $  \& $ \eqref{myeqn2} with $E_{g_{\theta}}[T]=\lambda(\theta)$. 
Assume that
 \begin{enumerate}[label=(\alph*)]
\item  $g(x,\theta) $ and the function $\lambda(\theta)$ are $k$-times differentiable for all $x \in A$ and $\theta \in \Theta$. \eqnum\label{myeqn3}
\item $0< \mathrm{Var}_{f_{\theta}}(S_{i}) < \infty$ and the covariance matrix  $\Sigma=(\Sigma_{ij})$ of the column vector $S=(S_{1},\cdots,S_{k})^{\intercal}$  is non-singular. \eqnum\label{myeqn4}
\item derivatives of functions of $\theta$ expressed as integrals with respect to $g(x,\theta)$ can be obtained by differentiating under the integral sign. \eqnum\label{myeqn5}
\end{enumerate}
Then for $T(X) \in \mathcal{C}_{\varphi}$, the variance of $T$ satisfies  
 \begin{eqnarray}
\mathrm{Var}_{f_{\theta}}(T(X)) \geq   M^{\intercal} \Sigma^{-1} M \label{eq:2ee}
\end{eqnarray}
where $M=(\lambda^{1}(\theta),\cdots, \lambda^{k}(\theta))^{\intercal}$,  $\lambda^{i}(\theta)$ is the $i^{\mathrm{th}}$ derivative of $\lambda$, $i=1,\cdots,k$.
\end{theorem}
\begin{proof}  Note that
$E_{f_{\theta}}[S_{i}]=0$ and $E_{f_{\theta}}[TS_{i}]= \lambda^{i}(\theta), \; \mathrm{where} \; i=1,\cdots ,k$. Hence $S_{i} \in \Psi$ and now apply  Proposition \ref{prop1} for $S_{i}$ to obtain the bound in Equation \eqref{eq:1e} with  $M=(\lambda^{1}(\theta),\cdots, \lambda^{k}(\theta))^{\intercal}$. \hfill $\blacksquare$
\end{proof}
\begin{remark} Note that by assuming appropriate regularity  assumptions the above theorem can also be obtained from  Theorem \ref{nonreg} as a limiting case.  
\end{remark}
\begin{remark} When $g=f$, Equation \eqref{eq:1e} reduces to the Bhattacharyya bounds of order $k$ given by Fraser and Guttman (1952) and for  $k=1$, it gives the Hammersley Chapman-Robbins bound. When $k=1$, Equation \eqref{eq:2ee} reduces to Naudts's generalized Cramer-Rao bound and when $g=f$, it reduces to the classical Bhattacharyya bounds of order $k$ in regular case. 
\end{remark}
\subsection{ Multiparameter case}
Let $X \sim f(x,\underaccent{\bar}{\theta})$, where $ \underaccent{\bar}{\theta} =(\theta_{1},\cdots,\theta_{p})^{\intercal} \in \Theta \in \Bbb R^{p}$. Let $T(X)$ be an unbiased estimator of a real valued function $\varphi(\underaccent{\bar}{\theta})$ of $\underaccent{\bar}{\theta}$.  Let $g(x,\underaccent{\bar}{\theta})$ be a density function parametrized by $\underaccent{\bar}{\theta}$ satisfying \eqref{myeqn1} $  \& $ \eqref{myeqn2}. Let the expectation of $T(X)$ with respect to $g(x,\theta)$ is $\lambda(\underaccent{\bar}{\theta})$, a real valued function of $\underaccent{\bar}{\theta}$, i.e. $E_{g_{\underaccent{\bar}{\theta}}}(T)=\lambda(\underaccent{\bar}{\theta})$. \\ 
Let $k \geq 1$ be an integer. Let $\mathbf{i}=(i_{1},\cdots,i_{p})$ such that $i_{j} \geq 0$, $0< i_{1}+\cdots+i_{p} \leq k$. 
Assume  the density function $g(x,\underaccent{\bar}{\theta})$ and   $\lambda(\underaccent{\bar}{\theta})$ have all partial derivatives with respect to $\theta_{1},\cdots, \theta_{p}$ of order up to $k$ and $k^{\mathrm{th}}$-order partial derivatives are continuous. Define
\begin{eqnarray}
\partial^{\mathbf{i}} :=\frac{\partial^{\mid \mathbf{i} \mid}}{\partial \theta^{i_{1}} \cdots \partial \theta^{i_{p}}}\quad \mathrm{where} \quad \mid \mathbf{i} \mid:= i_{1}+\cdots+i_{p}.
\end{eqnarray}
Define functions
\begin{eqnarray}
S^{\mathbf{i}} := \frac{1}{f(x,\underaccent{\bar}{\theta} )} \partial^{\mathbf{i}} g(x,\underaccent{\bar}{\theta}) ; \quad \lambda^{\mathbf{i}} (\underaccent{\bar}{\theta}) :=\partial^{\mathbf{i}} \lambda(\underaccent{\bar}{\theta})
\end{eqnarray}
We have the following theorem.
\begin{theorem} 
For  $T(X) \in \mathcal{C}_{\varphi}$, the variance of $T$ satisfies  
 \begin{eqnarray}
\mathrm{Var}_{f_{\underaccent{\bar}{\theta}}}(T(X)) \geq   M^{\intercal} \Sigma^{-1}(\underaccent{\bar}{\theta}) M \label{eq:11d}
\end{eqnarray}
where $\Sigma$ is the covariance matrix of the column vector $S=(S^{\mathbf{i}})$ containing all possible $S^{\mathbf{i}}$ and $M=( \lambda^{\mathbf{i}} (\underaccent{\bar}{\theta}))$ is a column  vector containing all possible  $\lambda^{\mathbf{i}}(\underaccent{\bar}{\theta})$.
\end{theorem}
\begin{proof}
Note that
 \begin{eqnarray}
E_{f_{\underaccent{\bar}{\theta}}}(S^{\mathbf{i}} ) &=& \int \partial^{\mathbf{i}} g(x,\underaccent{\bar}{\theta}) dx = 0 \\
\mathrm{Cov}_{f_{\underaccent{\bar}{\theta}}}(T,S^{\mathbf{i}} ) &= & \int T(x) \; \partial^{\mathbf{i}} g(x,\theta)\; dx = \lambda^{\mathbf{i}} (\underaccent{\bar}{\theta})
\end{eqnarray}
Hence for all $\mathbf{i}$, we have $S^{\mathbf{i}} \in  \Psi$. Now apply Proposition \ref{prop1} for  $S^{\mathbf{i}} \in  \Psi$ to obtain the bound.  \hfill $\blacksquare$
\end{proof}
\begin{remark} If $\mid \mathbf{i} \mid =1$, the bound in Equation \eqref{eq:11d} reduces to Naudts's bound in vector parameter. 
\end{remark}
\begin{note}  If the density $g(x,\underaccent{\bar}{\theta})$ is not regular, we can obtain an information inequality by replacing the partial derivatives by the corresponding divided difference formula. This is done as follows.\\
Consider a scalar function  $h(\underaccent{\bar}{\theta})$ of $\underaccent{\bar}{\theta}=( \theta_{1},\cdots,\theta_{p})$.  Let $k \geq 1$ be an integer. Let  us consider $k+1$ nodes of $\underaccent{\bar}{\theta}$ say, $\underaccent{\bar}{\theta}^{0}=( \theta_{1}^{0},\cdots,\theta_{p}^{0}),\cdots, \underaccent{\bar}{\theta}^{k}=(\theta_{1}^{k},\cdots,\theta_{p}^{k}) $. Define  
 \begin{eqnarray}
 \underaccent{\bar}{\theta}_{i}^{\nu}=( \theta_{1}^{\nu},\cdots,\theta_{i}^{\nu +1},\cdots  \theta_{p}^{\nu})
 \end{eqnarray}
Define the first divided difference of $h$ as
 \begin{eqnarray}
 \underset{\theta^{\nu+1}_{i}}{\Delta} h(\underaccent{\bar}{\theta} ^{\nu})=\frac{ h(\underaccent{\bar}{\theta}_{i}^{\nu}) - h(\underaccent{\bar}{\theta} ^{\nu})}{\theta_{i}^{\nu+1}-\theta_{i}^{\nu}}
 \end{eqnarray}
 where $\nu=0,\cdots,k-1$, $j=1$, $i=1,\cdots,p$. \\
In general, for $j \geq 2$, define $j^{\mathrm{th}}$ divided difference of $h$ as
 \begin{eqnarray}
 \underset{\theta_{i}^{\nu+1}, \cdots, \theta_{i}^{\nu+j}}{\Delta^{j}} h(\underaccent{\bar}{\theta^{\nu}}) & := & \underset{\theta_{i}^{\nu+j}}{\Delta} \cdots \underset{\theta_{i}^{\nu+1}}{\Delta}  h(\underaccent{\bar}{\theta^{\nu}})= \frac{\underset{\theta_{i}^{\nu+2}, \cdots, \theta_{i}^{\nu+j}}{\Delta^{j-1}} h(\underaccent{\bar}{\theta^{\nu +1}})- \underset{\theta_{i}^{\nu+1}, \cdots, \theta_{i}^{\nu+j-1}}{\Delta^{j-1}} h(\underaccent{\bar}{\theta^{\nu})}}{\theta_{i}^{\nu+j}-\theta_{i}^{\nu}} \nonumber \\
& & \; \mathrm{where} \; \nu= 0 \cdots k-j; \; j=1,\cdots,k; \;  i=1,\cdots,p
 \end{eqnarray}
\end{note} 
In many cases, one may be interested in estimating a vector valued function $\Phi(\underaccent{\bar}{\theta})$ of $\underaccent{\bar}{\theta}$. Let $\mathsf{T}=(T_{1},\cdots, T_{r})^{\intercal}$ be an unbiased estimator of $\Phi(\underaccent{\bar}{\theta}) =(\varphi_{1}(\underaccent{\bar}{\theta}),\cdots, \varphi_{r}(\underaccent{\bar}{\theta}))^{\intercal}$, where $r \leq p$. That is $E_{f_{\underaccent{\bar}{\theta}}}(T_{i})=\varphi_{i}(\underaccent{\bar}{\theta})$, $i=1,\cdots, r$. Let us consider $\mathsf{S}=(S_{1},\cdots, S_{m})^{\intercal}$, where functions  $S_{i} \in \Psi$, $i=1,\cdots,m.$ Let us assume that  the covariance matrix $\Sigma$ of $(r+m) \times 1$ vector $(\mathsf{T}, \mathsf{S})$ is positive definite. We have
\begin{equation}
\Sigma= \left[ \begin{array}{cc}
\Sigma_{\mathsf{T}} & \Sigma_{\mathsf{T} \mathsf{S}}  \\
                     \Sigma_{\mathsf{S} \mathsf{T}}  & \Sigma_{ \mathsf{S}}
\end{array} \right ]
 \end{equation} 
 where $ \Sigma_{\mathsf{T}}$ is the covariance matrix of $r \times 1$ vector $\mathsf{T}$, $ \Sigma_{\mathsf{S}}$ is the covariance matrix of $m \times 1$ vector $\mathsf{S}$ and $\Sigma_{\mathsf{T} \mathsf{S}} $ is the covariance  matrix between $\mathsf{T}$ and $\mathsf{S}$.
 If the covariance matrix $\Sigma_{\mathsf{S}}$ is invertible, the Shur complement of $\Sigma_{ \mathsf{S}}$ in $\Sigma$ is given by $\Sigma_{\mathsf{T}}-\Sigma_{\mathsf{T} \mathsf{S}} \Sigma^{-1}_{\mathsf{S}} \Sigma_{\mathsf{S} \mathsf{T}} $. It is easy to see that  $\Sigma_{\mathsf{T}}-\Sigma_{\mathsf{T} \mathsf{S}} \Sigma^{-1}_{\mathsf{S}} \Sigma_{\mathsf{S} \mathsf{T}} $ is positive definite since $\Sigma$ is positive definite. This can be written as
 \begin{eqnarray}
 \Sigma_{\mathsf{T}}-\Sigma_{\mathsf{T} \mathsf{S}} \Sigma^{-1}_{\mathsf{S}} \Sigma_{\mathsf{S} \mathsf{T}} \succ 0
 \end{eqnarray}
Equivalently, one can write 
 \begin{eqnarray}
 \Sigma_{\mathsf{T}} \succ \Sigma_{\mathsf{T} \mathsf{S}} \Sigma^{-1}_{\mathsf{S}} \Sigma_{\mathsf{S} \mathsf{T}}
 \end{eqnarray}
 which means that $\Sigma_{\mathsf{T}}-\Sigma_{\mathsf{T} \mathsf{S}} \Sigma^{-1}_{\mathsf{S}} \Sigma_{\mathsf{S} \mathsf{T}}$ is positive definite. \\
 The above inequality can be interpreted as follows. Consider a  linear estimator $\alpha^{\intercal} T$ which is unbiased for $\alpha^{\intercal} \Phi(\underaccent{\bar}{\theta})$. Then 
 \begin{eqnarray}
 \mathrm{Var}_{f_{\underaccent{\bar}{\theta}}} (\alpha^{\intercal} T) \geq \alpha^{\intercal} J(\underaccent{\bar}{\theta}) \alpha.
 \end{eqnarray}
where  $J(\theta)=\Sigma_{\mathsf{T} \mathsf{S}} \Sigma^{-1}_{\mathsf{S}} \Sigma_{\mathsf{S} \mathsf{T}}$.\\
 If  $\Sigma_{\mathsf{T}}=J(\theta)$, then variance of $\alpha^{\intercal} T$ attains this bound. That is, $\alpha^{\intercal} T$ is the minimum variance unbiased estimator for $\alpha^{\intercal} \Phi(\underaccent{\bar}{\theta})$ for any $\alpha$.
 \newpage
 \section{Discussions}
 In Proposition 5.2, Naudts (2004), it is shown that if $g$ is a deformed exponential family with a statistic $T$, then there exist an escort family $f$ such that variance of $T$ attains the Naudts's bound. Considering this from statistical perspective, let $f$ be the original model and assume that there exist a deformed exponential family $\mathcal{S}=\lbrace g(x;\theta) \mid \theta \in \Theta \subseteq \Bbb R \rbrace$ with a canonical statistic $T$ given by
 \begin{equation}
g(x;\theta)= Z(\theta T(x)-\phi(\theta)) \quad \text{or}  \quad F(g(x;\theta))= \theta T(x)-\phi(\theta) 
\end{equation} 
 where $F : (0,\infty) \longrightarrow \Bbb R$ is a smooth function satisfying $F'(x)>0$ \& $F''(x)< 0$, $Z$ is the inverse function  of $F$ and $\phi(\theta)$ is chosen such that $g$ is a probability density function. \\
 Then assume that $f$ is the $F$-escort distribution of $g$ given by
 \begin{align}
f(x,\theta)= \frac{1}{F'(g)h_{F}(\theta)}; \quad  h_{F}(\theta)=\int \frac{1}{F'(g(x;\theta))} dx  
\end{align}
 Then it is easy to see that $T$ is an unbiased estimator of the expectation parameter $\eta=E_{f_{\theta}}(T)$ and the variance of $T$  attains the Naudts's bound. 
 
For the geometric interpretation, we first  observe that  if $g$ is an exponential family then the original model $f$ is equal to $g$. Then the estimator $T$ is an unbiased estimator of the expectation parameter and its variance attains the Cramer-Rao lower bound. The exponential family has a dually flat structure. The expectation parameter is the dual coordinate in the dually flat $\alpha$-geometry  by Amari (1985) with $\alpha=1$. When $g$ is a deformed exponential family with dually flat $\chi$-geometric structure, then $E_{f_{\theta}}(T)$ is the dual coordinate (Refer Amari et al. (2012) for more details).  As observed above, $T$ is  an unbiased estimator for $E_{f_{\theta}}(T)$ and its variance attains the Naudts's lower bound. In the context of statistical inference, the $\chi$-geometry seems to provide a useful generalization of $1$-geometry.  

\section*{Acknowledgment}
The first named author is supported by the Postdoctoral Fellowship from Indian Institute of Technology Bombay, Mumbai.

\end{document}